\documentclass[12pt]{amsart}

\usepackage{graphicx} 
\usepackage{amsfonts}
\usepackage{amscd}
\usepackage{amssymb}
\usepackage{alltt}
\usepackage{txfonts}
\usepackage{mathrsfs}

\newtheorem{thm}{Theorem}
\newtheorem{cor}[thm]{Corollary}
\newtheorem{lem}[thm]{Lemma}
\newtheorem{prop}[thm]{Proposition}

\newtheorem*{thm11}{Theorem 1.1}
\numberwithin{equation}{section}
\numberwithin{thm}{section}
\DeclareMathOperator{\RE}{Re}

\DeclareMathOperator{\spc}{ }

\DeclareMathOperator{\GL}{GL}

\DeclareMathOperator{\Gal}{Gal}

\DeclareMathOperator{\End}{End}
\DeclareMathOperator{\tensor}{\otimes}
\DeclareMathOperator{\alg}{alg}
\DeclareMathOperator{\an}{an}
\DeclareMathOperator{\cl}{cl}
\DeclareMathOperator{\semisimp}{ss}
\DeclareMathOperator{\Aut}{Aut}
\DeclareMathOperator{\ord}{ord}

\newcommand{\To}{\longrightarrow}

\newcommand{\R}{\mathbb{R}}
\newcommand{\C}{\mathbb{C}}
\newcommand{\Q}{\mathbb{Q}}
\newcommand{\Z}{\mathbb{Z}}

\newcommand{\F}{\mathbb{F}}
\newcommand{\abs}[1]{\left\vert#1\right\vert}
\newcommand{\set}[1]{\left\{#1\right\}}
\newcommand{\norm}[1]{\left\Vert#1\right\Vert}

\title[A rank inequality for the Tate Conjecture]{A rank inequality for the Tate Conjecture over global function fields}
\author{Christopher Lyons}
\date{March 2008}

\begin{document}

\maketitle
   
We present an observation of D. Ramakrishnan concerning the Tate Conjecture for varieties over a global function field (i.e., the function field of a smooth projecture curve over a finite field), which was pointed out during a lecture given at the AIM's workshop on the Tate Conjecture in July 2007.  The result is perhaps ``known to the experts," but we record it here, as it does not appear to be in print elsewhere.  We use the global Langlands correspondence for the groups $\GL_n$ over global function fields, proved by L. Lafforgue \cite{Laf}, along with an analytic result of H. Jacquet and J. Shalika \cite{JS} on automorphic $L$-functions for $\GL_n$.  Specifically, we use these to show (see Theorem 1.1 below) that, for a prime $\ell\neq \text{char } k$, the dimension of the subspace spanned by the rational cycles of codimension $m$ on our variety in its $2m$-th $\ell$-adic cohomology group (the so-called algebraic rank) is bounded above by the order of the pole at $s=m+1$ of the associated $L$-function (the so-called analytic rank).  The interest in this result lies in the fact that, with the exception of some special instances like certain Shimura varieties and abelian varieties which are potentially CM type, the analogous result for varieties over number fields is still unknown in general, even for the case of divisors ($m=1$).


\section{Preliminaries and Main Result}\label{prelim_sec}

Tate's original article \cite{T1} serves as a good reference for this section, and also gives insight into the motivation behind the conjectures.  The similar case of varieties over $\Q$, which has the additional advantage that singular cohomology and Hodge theory can be brought to bear on the problem, is discussed in \S1 of \cite{R1}.

Let $X$ be a smooth, projective, geometrically connected variety over a global function field $k$.  Let $\F_q$ denote the constant field of $k$ and $\bar k$ its separable closure.  Fix a prime $\ell\neq \text{char } k$.  For an integer $0\leq m\leq \dim X$, write 
\[
V_\ell=H_{\text{\'{e}t}}^{2m}(X\times_k\bar k,\Q_\ell)
\]
for the $2m$-th $\ell$-adic cohomology group, which is a finite-dimensional vector space over $\Q_\ell$.  The natural action of $\Gamma_k:=\Gal(\bar k/k)$ on $\bar k$ gives an action of $\Gamma_k$ on $X\times_k\bar k$, which in turn gives rise to a continuous linear action of $\Gamma_k$ on $V_\ell$.  Thus we get a continuous representation $\rho_\ell\colon\Gamma_k\to\Aut_{\Q_\ell}(V_\ell)$.  Moreover, for almost every place $v$ of $k$ (i.e., for all but a finite number), $\rho_\ell$ is \emph{unramified} at $v$, in the sense that the inertia subgroup $I_v$ of any decomposition group $D_v$ for $v$ is in the kernel of $\rho_\ell$.

To this representation $\rho_\ell$ of $\Gamma_k$ can be associated an $L$-function $L(\rho_\ell,s)$; we will not need the full $L$-function, but rather the incomplete form $L^S(\rho_\ell,s)$, where $S$ is any finite set of places containing those where either $\rho_\ell$ is ramified or $X$ has bad reduction.  By definition,
\[
L^S(\rho_\ell,s)=\prod_{v\notin S} L_v(\rho_\ell,s),
\]
where
\[
L_v(\rho_\ell,s)=\det\bigl(1-q_v^{-s}\rho_\ell(Fr_v)\bigr)^{-1}
\]
for any $v\notin S$.  Here $Fr_v$ is the \emph{geometric} Frobenius conjugacy class of $v$ in $\Gamma_k$ and $q_v$ is the residue cardinality of $v$.  Then by the proof of the Weil Conjectures \cite{De1}, we have $L_v(\rho_\ell,s)=Z_v(q_v^{-s})^{-1}$, where $Z_v(T)$ is a polynomial with coefficients in $\Z$ which factors as
\[
Z_v(T)=\prod_{i=1}^{b} (1-\alpha_{i,v}T),
\]
where $b=\dim_{\Q_\ell} V_\ell$ and each $\alpha_{i,v}$ has absolute value $q_v^m$ under any complex embedding.  (Note that the $\alpha_{i,v}$ are the eigenvalues of $\rho_\ell(Fr_v)$.)  It follows that the Euler product $L^S(\rho_\ell,s)$ converges absolutely for $\RE(s)>m+1$, and in fact uniformly on compact subsets, giving a holomorphic function in this half-plane.

Now let $\mathcal C^m$ denote group of cycles of codimension $m$ on $X$, which is the free abelian group generated by closed irreducible subvarieties of codimension $m$ on $X\times_{k} \bar k$.  Let
\[
V_\ell(m):=V_\ell\tensor_{\Q_\ell} \Q_\ell(m);
\]
here we set
\[
\Q_\ell(1):=\biggl(\varprojlim_{j} \spc \mu_{\ell^j}\biggr) \tensor_{\Z_\ell}\Q_\ell,
\]
with the action of $\Gamma_k$ given by its action on each $\mu_{\ell^j}$, the group of $\ell^j$th roots of unity of in $\bar k$, and then we take $\Q_\ell(m):=\Q_\ell(1)^{\otimes m}$.   (One calls $V_\ell(m)$ the $m$th \emph{Tate twist} of $V_\ell$.)  One can show (see \cite{Mil}, VI.9)  the existence of a canonical cycle class map
\[
\cl_m\colon\mathcal C^m\to V_\ell(m).
\]
There is a natural $\Gamma_k$-action on $\mathcal C^m$ coming from that on $X\times_k \bar k$, and it turns out that $\cl_m$ is a morphism of $\Gamma_k$-modules (i.e., is a $\Gamma_k$-equivariant map).  This means that a cycle in $(\mathcal C^m)^{\Gamma_k}$ maps into $V_\ell(m)^{\Gamma_k}$.

Define the following quantities:
\begin{align*}\label{three}
r_{\alg,k}^{(m)} &= \dim_{\Q_\ell} \Bigl[\cl_m\Bigl((\mathcal C^m)^{\Gamma_k}\Bigr)\otimes \Q_\ell\Bigr], \\
r_{\ell,k}^{(m)} &= \dim_{\Q_\ell} V_\ell(m)^{\Gamma_k}, \tag{1a}\\
r_{\an,k}^{(m)} &= -\text{ord}_{s=m+1} L^S(\rho_\ell,s).
\end{align*}
(If $L^S(\rho_\ell,s)$ is known to have meromorphic continuation to the point $s=m+1$, this last quantity makes sense as the order of pole at $s=m+1$; otherwise we take it to be the unique integer $a$, if it exists, such that
\[
\lim_{s\to m+1}(s-m-1)^a L^S(\rho_\ell,s)
\]
is finite and nonzero.  Also note that $r_{\an,k}^{(m)}$ is independent of our choice of $S$ by Deligne's proof of the Weil Conjectures, as long as $S$ satisfies the aforementioned conditions.) The first and last quantities are referred to as the \emph{algebraic} and \emph{analytic} ranks, respectively.  The $\Gamma_k$-equivariance of $\cl_m$ above gives that
\[
r_{\alg,k}^{(m)}\leq r_{\ell,k}^{(m)}.
\]
J. Tate's conjecture \cite{T1} is that, in fact, all three quantities in (\ref{three}) are equal.

In \S\ref{proof_sec} we will show\footnote{In the published version, the statement of this theorem contained an error. We have corrected it here and we thank Uwe Jannsen and Dinakar Ramakrishnan for bringing it to our attention.}

\begin{thm11}
For a smooth, projective, geometrically connected variety $X$ over a global function field $k$, we have $$ r_{\ell,k}^{(m)}\leq r_{\an,k}^{(m)},$$ and thus
\[
r_{\alg,k}^{(m)}\leq r_{\an,k}^{(m)},
\]
for any $0\leq m\leq \dim X$.
\end{thm11}

We remark that one knows $r_{\alg,k}^{(m)}\geq 1$; indeed, thinking of $X$ as embedded in some projective space over $k$, the (nonzero) cycle class of a hyperplane section is defined over $k$, and its $m$-fold cup product gives a nonzero cycle class in $V_\ell(m)$.  Thus the theorem gives $r_{\an,k}^{(m)}\geq 1$ as well.


\section{Automorphic representations of $\GL_n(\mathbb A_k)$ and their $L$-functions}\label{autom_sec}

Our strategy in proving that $r_{\ell,k}^{(m)}=r_{\an,k}^{(m)}$ is to use Lafforgue's result that the representation $\rho_\ell$ is \emph{modular}; that is to say, there is an automorphic representation of $\GL_n(\mathbb A_k)$ whose $L$-function has the same analytic behavior as that of $\rho_\ell$.  This is fortuitous, since the analytic behavior of automorphic $L$-functions is a priori much better understood than that of $L$-functions of Galois representations such as $\rho_\ell$.  For this reason, we take this to section to briefly recall facts about $L$-functions of cuspidal automorphic representations.  We refer the reader to \S1.2 of \cite{R1} or \S1.1 of \cite{Lau} for a more thorough introduction.

With $k$ still being a global function field, let $\mathbb A_k$ denote its ring of adeles, and let $\omega$ denote a \emph{unitary} idele class character of $k$.  We define a space of functions
\[
L^2(\omega):= L^2(\GL_n(k)Z(\mathbb A_k)\backslash\GL_n(\mathbb A_k),\omega),
\]
where $Z(\mathbb A_k)\simeq \mathbb A_k^\times$ denotes the center of $\GL_n(\mathbb A_k)$, as the (classes of) measurable functions $\phi\colon\GL_n(\mathbb A_k)\to \C$ which satisfy
\begin{itemize}
\item $\phi(\gamma g z)=\omega(z)\phi(g)$ for all $\gamma\in\GL_n(k)$, $g\in \GL_n(\mathbb A_k)$, and $z\in Z(\mathbb A_k)$,
\item $\int_{\GL_n(k)Z(\mathbb A_k)\backslash\GL_n(\mathbb A_k)}\abs{\phi(g)}^2 dg<\infty$;
\end{itemize}
note that the second condition makes sense, since the first condition and $\omega$ being unitary allow $\abs{\phi}$ to descend to a function on $\GL_n(k)Z(\mathbb A_k)\backslash\GL_n(\mathbb A_k)$.  There is a subspace $L_{\text{cusp}}^2(\omega)$ of $L^2(\omega)$ of those functions $\phi$ satisfying the following condition: if $U$ is the unipotent radical of any standard parabolic subgroup of $\GL_n$, then we have
\[
\int_{U(k)\backslash U(\mathbb A_k)} \phi(u g) du=0
\]
for almost all $g\in \GL_n(\mathbb A_k)$.  This subspace $L^2_{\text{cusp}}(\omega)$ is referred to as the space of \emph{cusp forms} on $\GL_n(\mathbb A_k)$ of \emph{central character $\omega$}.

We have a left action of $\GL_n(\mathbb A_k)$ on $L^2(\omega)$ by right translations (that is, by the action $(h\cdot \varphi)(g):=\varphi(gh)$ for $h\in \GL_n(\mathbb A_k)$).  This action happens to preserve $L_{\text{cusp}}^2(\omega)$, and thus $L_{\text{cusp}}^2(\omega)$ yields a complex representation of $\GL_n(\mathbb A_k)$.  This representation comes with a number of desirable properties: in particular, we have a semisimple decomposition
\[
L_{\text{cusp}}^2(\omega) \simeq \widehat\bigoplus_\pi V_\pi^{m_\pi},
\]
where $(\pi,V_\pi)$ runs over a system of representatives for isomorphism classes of irreducible \emph{admissible} complex representations of $\GL_n(\mathbb A_k)$.  Furthermore, the \emph{multiplicity one theorem for $\GL_n$} of Shalika says that, for any such $\pi$, we have either $m_\pi=1$ or $m_\pi=0$.  We define a \emph{cuspidal automorphic representation} of $\GL_n(\mathbb A_k)$ (or simply a \emph{cuspidal representation}) with central character $\omega$ to be any component $(\pi,V_\pi)$ of this direct sum for which $m_\pi=1$.

Now let $\omega$ be an arbitrary idele class character of $\mathbb A_k^\times/k^\times$, which is not necessarily unitary.  Let $\norm{\cdot}_{\mathbb A_k}$ denote the adelic norm on $\mathbb A_k$.  Then there is a unique $t\in\R$ and a unique unitary idele class character $\omega_0$ such that
\[
\omega=\omega_0\norm{\cdot}_{\mathbb A_k}^t.
\]
One may take the definition of a cuspidal representation $\pi$ of $\GL_n(\mathbb A_k)$ with central character $\omega$ to be one of the form
\[
\pi:= \pi' \otimes (\norm{\cdot}_{\mathbb A_k}^t\circ \det),
\]
where $\pi'$ is a cuspidal representation of $\GL_n(\mathbb A_k)$ with central character $\omega_0$, as defined above.  From now on, we use the term ``cuspidal representation'' in this sense, with no restriction on the central character unless otherwise specified.

For each cuspidal representation $\pi$, it turns out that $\pi\simeq\bigotimes_v^\prime \pi_v$, which is a restricted tensor product that runs over the places $v$ of $k$.  Each factor $(\pi_v,V_{\pi_v})$ is a complex representation of $\GL_n(k_v)$ which is irreducible and admissible.  Let $\mathcal O_v$ be the ring of integers in $k_v$ and let $K_v=\GL_n(\mathcal O_v)$.  We say that $\pi_v$ is \emph{unramified} if $V_{\pi_v}^{K_v}$ is nontrivial.  For cuspidal $\pi$, one knows $\pi_v$ is unramified for almost every $v$.

Inspired by a theorem of I. Satake, R. Langlands attached to any unramified irreducible admissible complex representation $\pi_v$ of $\GL_n(k_v)$ an unordered $n$-tuple $\bigl\{\beta_{1,v},\beta_{2,v},\dotsc,\beta_{n,v}\bigr\}$ of nonzero complex numbers.  These numbers, called the \emph{Langlands parameters} (or just the \emph{parameters}) of $\pi_v$, determine $\pi_v$ up to isomorphism.  Hence a cuspidal representation $\pi$ determines such an $n$-tuple for all $v$ at which $\pi$ is unramified.

One fact needed below is that, if $\lambda$ is an idele class character and $\pi'$ is a cuspidal representation, then the representation
\[
\pi:=\pi'\otimes(\lambda\circ\det)
\]
is also cuspidal. Furthermore,  if $v$ is a place such that $\lambda_v$ is unramified and $\pi_v'$ is unramified with parameters $\bigl\{\beta_{j,v}\bigr\}$, then $\pi_v$ is also unramified and has parameters $\bigl\{\beta_{j,v}\lambda(\varpi_v)\bigr\}$, where $\varpi_v$ is a uniformizer for $k_v$.

Given $\pi$, it is known (\cite{JS},\cite{JPSS}) that knowledge of the parameters of $\pi_v$ for almost every unramified $v$ is enough to determine $\pi$ up to isomorphism, as long as $\pi$ has unitary central character (which is always true after an appropriate twist by $\norm{\cdot}_{\mathbb A_k}^t$, $t\in\R$):

\begin{thm}[``strong mulitiplicity one"; Jacquet, I. Piatetski-Shapiro, Shalika]
Suppose $\pi_1$ and $\pi_2$ are two cuspidal representations, both with unitary central character, satisfying $\pi_{1,v}\simeq\pi_{2,v}$ for all $v$ outside some finite set $S$ of places of $k$.  Then $\pi_1\simeq\pi_2$.
\end{thm}

If $\pi_v$ is unramified, define
\[
L_v(\pi,s)=\bigl[(1-\beta_{1,v}q_v^{-s})\dotsm(1-\beta_{n,v}q_v^{-s})\bigr]^{-1}
\]
and let
\[
L^S(\pi,s) = \prod_{v\notin S} L_v(\pi,s)
\]
be the incomplete $L$-function associated to $\pi$, where $S$ is a finite set of places containing those at which $\pi$ is ramified.  Then in \cite{JS} (see Propositions 3.3 and 3.6) the following result is proved:

\begin{thm}[Jacquet, Shalika]\label{JacSha}
Suppose that $\pi$ has unitary central character.  Then $L^S(\pi,s)$ is holomorphic for $\RE(s)>0$ if $\pi$ is not an idele class character of the form $\norm{\cdot}_{\mathbb A}^{it}$, $t\in\R$.
\end{thm}

On the other hand, when $\pi=\norm{\cdot}_{\mathbb A}^{it}$, so that $n=1$ and $\pi_v$ is unramified everywhere, we have $\beta_{1,v}=q_v^{-it}$ for all $v$.  Hence in this case, $L^S(\pi,s)$ is  the translated Dedekind zeta function $\zeta_k(s+it)$ of $k$ (divided by a finite number of Euler factors if $S\neq\emptyset$), which is holomorphic in $\C$ except for a simple pole at $s=1-it$.  In particular, we have

\begin{cor}\label{ord}
Suppose that $\pi$ has unitary central character.  Then
\[
-\ord_{s=1}L^S(\pi,s)=\begin{cases} 
1&\text{if } \pi \text{ trivial}\\ 
0&\text{if } \pi \text{ nontrivial}\end{cases}. \tag{3a}
\]
\end{cor}


\section{$\ell$-adic representations and the Langlands correspondence over $k$ for $\GL_n$}\label{ladic_sec}

Lafforgue's result pairs each irreducible $\ell$-adic Galois representation with a cuspidal representation.  We will describe the objects on the first side more explicitly, and then describe the correspondence.  The survey \cite{Lau} is a good reference for this material, notably \S1.2 and \S1.3.  We then give an easy extension of this result.

For any $n\geq 1$, we will define an \emph{$n$-dimensional $\ell$-adic representation} of $\Gamma_k$ to be a continuous homomorphism $\sigma_\ell\colon\Gamma_k\to \Aut_{\bar\Q_\ell}(M)$ for some finite-dimensional vector space $M$ over $\bar\Q_\ell$.  Let $\mathcal G_n'$ denote a system of representatives for the isomorphism classes of \emph{irreducible} $n$-dimensional $\ell$-adic representations $\sigma_\ell$ of $\Gamma_k$ which satisfy the following three additional properties:
\begin{enumerate}
\item[(i)] There is a basis of $M$ such that, when using this basis to identify $\Aut_{\bar\Q_\ell}(M)$ with $\GL_n(\bar\Q_\ell)$, one has $\sigma_\ell(\Gamma_k)\subseteq \GL_n(E)$ for some finite extension $E\subseteq\bar\Q_\ell$ of $\Q_\ell$.
\item[(ii)] There are only a finite number of places $v$ of $k$ at which $\sigma_\ell$ is ramified, in the sense described in \S\ref{prelim_sec}.
\item[(iii)] The character $\det \sigma_\ell$ is of finite order.
\end{enumerate}

At this point, we fix once and for all an isomorphism $\iota\colon\bar \Q_\ell \to \C$. To any such $\sigma_\ell$ we can assign an incomplete $L$-function $L^S(\sigma_\ell,s)$, for a finite set $S$ containing the ramified places of $\sigma_\ell$, in exactly the same manner as in \S\ref{prelim_sec}: for $v\notin S$, set
\[
\label{galoislocalfactor}
L_v(\sigma_\ell,s)=\det\bigl(1-q_v^{-s}\sigma_\ell(Fr_v)\bigr)^{-1} \tag{4a}
\]
and then set
\[
L^S(\sigma_\ell,s)=\prod_{v\notin S} L_v(\sigma_\ell,s).
\]
Thanks to the isomorphism $\iota$, we view this as a complex $L$-function.

Let $\mathcal A_n'$ denote a system of representatives for the isomorphism classes of cuspidal representations of $\GL_n(\mathbb A_k)$ with \emph{finite order} central character.  Then the following \emph{global Langlands correspondence for $\GL_n$} was proved for the case $n=2$ by V. Drinfeld \cite{Dr1},\cite{Dr2} and later extended to all cases $n>2$ by Lafforgue \cite{Laf} (with the case $n=1$ following from class field theory for $k$):

\begin{thm}[Lafforgue]\label{Lafforgue}
There is a unique bijection $\mathcal G_n'\to\mathcal A_n'$, $\sigma_\ell\mapsto\pi$, such that for almost every place $v$ at which $\sigma_\ell$ and $\pi$ are unramified,
\[
L_v(\sigma_\ell,s)=L_v(\pi,s).
\]
\end{thm}

We now discuss how to extend this theorem to the case where the ``finite order" restrictions are removed from the definitions of $\mathcal G_n'$ and $\mathcal A_n'$.  This extension is something which is presumably well-known to experts, but does not seem to be written down.  The key ingredient is the description of unramified (Galois and idele class) characters for $k$ given by class field theory in the function field setting.

Define $\mathcal A_n$ to be a system of representatives for the isomorphism classes of cuspidal representations of $\GL_n(\mathbb A_k)$ (with no restriction on the central character).  Also let $\mathcal G_n$ be defined exactly as $\mathcal G_n'$ above, but without condition (iii), and define $L_v(\sigma_\ell,s)$ using (\ref{galoislocalfactor}) if $\sigma_\ell\in\mathcal G_n$ is unramified at $v$.  Then we have the following:

\begin{cor}\label{easycor}
There is a unique bijection $\mathcal G_n\to\mathcal A_n$, $\sigma_\ell\mapsto\pi$, such that for almost every place $v$ at which $\sigma_\ell$ and $\pi$ are unramified,
\[
\label{localequality}
L_v(\sigma_\ell,s)=L_v(\pi,s). \tag{4b}
\]
\end{cor}

(We note that it is this bijection which is stated in the papers of Drinfeld.  The finite-order assumptions are only present in Lafforgue's work, and are not serious obstacles, as this corollary demonstrates.)

Before getting to the proof of this corollary, we need the following result:

\begin{lem}\label{charlem}
Let $E$ be a finite extension of $\Q_\ell$ and let $\chi\colon\Gamma_k\to E^\times$ be a continuous character.  Then there is a finite power of $\chi$ which is unramified everywhere.
\end{lem}

We remark that this statement is false for number fields, due mainly to the presence of archimedean places.  (See \S\ref{CM_sec} for more details.)

\medskip

\noindent\emph{Proof of \ref{charlem}.} By compactness of $\Gamma_k$, we may assume $\chi$ takes values in $\mathcal O_E^\times\subseteq E^\times$ by changing basis (\cite{Se}, p.1).  We have an isomorphism
\[
\mathcal O_E^\times \simeq \mu_E \times \mathcal O_E,
\]
where $\mu_E$ is the group of roots of unity in $E$.  If $\ell^r$ is the cardinality of the residue field of $E$, then $\mu_E$ is cyclic of order $\ell^r-1$, while $\mathcal O_E$ is a pro-$\ell$ group.  

Now let $v$ be any place of $k$.  By local class field theory, if $I_v$ is the inertia subgroup of any decomposition group $D_v \subseteq \Gamma_k$ of $v$, then the image of $I_v$ in the abelianization $\Gamma_k^{ab}$ of $\Gamma_k$ is the product of a finite cyclic group and a pro-$p$ group, where $p=\text{char }k\neq \ell$.  Since $\chi$ factors through $\Gamma_k^{ab}$, this forces $\chi(I_v)\subseteq\mu_E\times\set{0}$ and shows that $\chi^{\ell^r-1}$ is unramified. \qed

\medskip

\noindent\emph{Proof of \ref{easycor}.} Pick any $\sigma_\ell\in\mathcal G_n$, and suppose, by (i), that $\sigma_\ell$ takes values in $\GL_n(E)$ for a finite extension $E$ of $\Q_\ell$.

The character $\chi=\det\sigma_\ell$ is continuous and takes values in $E^\times$, so by the lemma we pick $w\in\Z$ such that $\chi^w$ is unramified.  By global class field theory for $k$ (see \cite{AT}, p.76), this means that $\chi^w$ factors through $\Gal(k\bar\F_q/k)\simeq\Gal(\bar \F_q/\F_q)\simeq \hat\Z$ (recall that $\F_q$ is the constant field of $k$) and is completely determined by the image of $1\in\hat\Z$.  Denoting this element as $\chi^w(1)$ by abuse of notation, we choose some $z\in\bar\Q_\ell$ such that $z^{wn}=\chi^w(1)$.

Let $\lambda_\ell\colon\Gamma_k\to E(z)^\times$ be the unique unramified character such that, again by abuse of notation, $\lambda_\ell(1)=z$ and thus $\lambda_\ell^{wn}=\chi^w$.  By global class field theory, $\lambda_\ell$ corresponds to an unramified idele class character $\lambda\colon\mathbb A_k^\times/k^\times\to\C^\times$, in the sense that $\lambda_\ell(Fr_v)=\lambda(\varpi_v)$ for all $v$, where $\varpi_v$ is a uniformizer of $k_v$.  (Note that this is the opposite convention of that in \cite{AT}, since $Fr_v$ is the geometric Frobenius.  Also recall we have identified $\bar \Q_\ell$ with $\C$ via the fixed isomorphism $\iota$.)

Since $\sigma_\ell$ is unramified almost everywhere, $\sigma_\ell\otimes\lambda_\ell^{-1}$ is a continuous representation of $\Gamma_k$ which also is unramified almost everywhere and that takes values in $\GL_{n}(E(z))$.  Furthermore,
\[
(\det(\sigma_\ell \tensor \lambda_\ell^{-1}))^{w}=(\chi\lambda_\ell^{-n})^w=1,
\]
i.e., the determinant of $\sigma_\ell\otimes\lambda_\ell^{-1}$ has finite order.  Thus Theorem \ref{Lafforgue} gives a unique cuspidal representation $\pi'$, with central character of finite order, such that
\[\label{4c}
L_v(\pi',s)=L_v(\sigma_\ell\otimes \lambda_\ell^{-1},s)\tag{4c}
\]
for almost all $v$.  

Let $S$ be a finite set of places containing those for which (\ref{4c}) does not hold, as well as the ramified places of $\pi'$ and $\sigma_\ell$.  For $v\notin S$, (\ref{4c}) means that the parameters $\bigl\{\beta_{j,v}\bigr\}$ of $\pi_{v}'$ coincide with the eigenvalues of
\[
\bigl(\sigma_\ell\otimes\lambda_\ell^{-1}\bigr)(Fr_v)=\sigma_\ell(Fr_v)\lambda_\ell(Fr_v)^{-1}.
\]
This implies that the parameters $\bigl\{\beta_{j,v}\lambda(\varpi_v)\bigr\}$ of the cuspidal representation $\pi=\pi'\tensor(\lambda\circ\det)$ coincide with the set of eigenvalues of $\sigma_\ell(Fr_v)$, and therefore that
\[
L_v(\pi,s)=L_v(\sigma_\ell,s)
\]
for $v\notin S$.

Let us denote this construction of $\pi$ from $\sigma_\ell$ as $r_n\colon\sigma_\ell\mapsto\pi$.  We have verified that almost all local $L$-factors of $\sigma_\ell$ and $\pi$ agree, as required in the statement of the corollary.  We now verify that $r_n$ satisfies the other necessary properties.

\emph{$r_n$ is well-defined:} The only potential ambiguity in our construction is the choice of $z$ such that $z^{wn}=\chi^w(1)$, and hence the choice of $\lambda_\ell$.  Suppose that $\tilde\lambda_\ell$ were another valid choice, corresponding to the idele class character $\tilde\lambda$.  Then the representations $\pi'$ and $\tilde\pi'$ associated to $\sigma_\ell\otimes\lambda_\ell^{-1}$ and $\sigma_\ell\otimes\tilde\lambda_\ell^{-1}$, respectively, may indeed differ.  However, the representations $\pi'\otimes(\lambda\circ\det)$ and $\tilde\pi'\otimes(\tilde\lambda\circ\det)$ will be the same, as one can verify by comparing their parameters and using the strong multiplicity one theorem.  Hence $\pi$ is defined unambiguously.

\emph{$r_n$ is injective:} Here one uses the fact that knowledge of almost every local factor $L_v(\sigma_\ell,s)$ determines $\sigma_\ell$ up to isomorphism, essentially by Chebotarev density  (see the theorem on p.I-10 of \cite{Se}, which applies to all global fields).

\emph{$r_n$ is surjective:} Pick $\pi\in\mathcal A_n$ with central character $\omega$.  Then
\[
\omega=\omega_f\norm{\cdot}_{\mathbb A_k}^y
\]
for a finite order character $\omega_f$ and some $y\in\C$.  Indeed, because $k$ is a function field, this follows from the fact that the kernel of $\norm{\cdot}_{\mathbb A_k}$ is compact and countable, and so its complex characters are all of finite order, as well as the fact that the image of $\norm{\cdot}_{\mathbb A_k}$ is isomorphic to $\Z$.  Let $\lambda=\norm{\cdot}_{\mathbb A_k}^y$, which is an unramified idele class character, and let $\lambda_\ell\colon\Gamma_k\to\bar\Q_\ell^\times$ be the corresponding unramified $\ell$-adic character, in the sense described above.  Then $\pi\otimes(\lambda\circ\det)^{-1}\in\mathcal A_n'$ corresponds, by the theorem, to a representation $\sigma_\ell'\in\mathcal G_n'$, and one checks that this implies
\[
L_v(\pi,s)=L_v(\sigma_\ell'\otimes\lambda_\ell,s)
\]
for almost every $v$.  Thus $\pi$ corresponds to $\sigma_\ell:=\sigma_\ell'\otimes\lambda_\ell$.

\emph{$r_n$ is the unique bijection satisfying (\ref{localequality}) for almost all $v$:}  Were there another bijection with this property, we would wind up with two nonisomorphic cuspidal representations $\pi_1$, $\pi_2$ whose parameters match at almost every place $v$.  Thus, for some place $v$, we have an isomorphism $\pi_{1,v}\simeq\pi_{2,v}$ of unramified representations of $\GL_n(k_v)$.  This implies the central characters of $\pi_{1,v}$ and $\pi_{2,v}$ are both equal to $|\cdot|_v^z$ for some $z\in\C$; here, $|\cdot|_v$ is the normalized absolute value on $k_v$.  It follows that if
\[
\pi_i':=\pi_i\otimes (\norm{\cdot}_{\mathbb A_k}^{-\RE \spc z} \circ \det)
\]
for $i=1,2$, then each $\pi_i'$ has unitary central character and $\pi_{1,v}'\simeq\pi_{2,v}'$ for almost every $v$.  By the strong multiplicity one theorem, this gives $\pi_1'\simeq \pi_2'$, and hence $\pi_1\simeq \pi_2$, a contradiction.  So $r_n$ must be the unique bijection with the given property.\qed


\section{Proof of Theorem 1.1}\label{proof_sec}

Recall the setup and notation in \S\ref{prelim_sec}.  We have now reviewed the tools needed to prove our main result:

\begin{thm11}
For a smooth, projective, geometrically connected variety $X$ over a global function field $k$, we have $$ r_{\ell,k}^{(m)}\leq r_{\an,k}^{(m)},$$ and thus
\[
r_{\alg,k}^{(m)}\leq r_{\an,k}^{(m)},
\]
for any $0\leq m\leq \dim X$.
\end{thm11}

\begin{proof}
Let $\rho_\ell(m)\colon\Gamma_k\to\Aut_{\Q_\ell}V_\ell(m)$ denote the $m$-th Tate twist of $\rho_\ell$.

The semisimplification of the extension of $\rho_\ell(m)$ to an action on $V_\ell(m)\otimes\bar\Q_\ell$ is a direct sum of irreducible $\bar\Q_\ell$-representations.  An easy exercise shows the existence of a finite extension $E/\Q_\ell$ over which this semisimple decomposition is defined.  In other words, we have
\[
(V_\ell(m)\otimes E)^{\semisimp} =\bigoplus_i M_i
\]
where each $M_i$ is an $E$-vector space such that $\Gamma_k$ acts irreducibly on $M_i\otimes\bar \Q_\ell$ (and hence irreducibly on $M_i$) via the extension of $\rho_\ell(m)$.

Let $\rho_i\colon\Gamma_k\to \Aut_{\bar\Q_\ell}(M_i\otimes \bar\Q_\ell)$ denote the irreducible $\bar\Q_\ell$-representation defined by $\rho_\ell(m)$.  Recall from \S\ref{prelim_sec} that, because $\rho_\ell(m)$ arises from the cohomology of $X$, it is unramified at almost every place of $k$; thus $\rho_i$ inherits this property as well.  Hence, because $\rho_i$ is defined over a finite extension $E/\Q_\ell$ as just remarked, it follows that $\rho_i\in \mathcal G_{n_i}$ in the notation of \S\ref{ladic_sec}, where $n_i=\dim M_i$.  By Corollary \ref{easycor}, there is a unique cuspidal representation $\pi_i\in\mathcal A_{n_i}$ such that
\[\label{5a}
L_v(\pi_i,s)=L_v(\rho_i,s) \tag{5a}
\]
for almost every $v$.

Recall from \S\ref{prelim_sec} that the eigenvalues of almost every $\rho_\ell(Fr_v)$ are algebraic and have absolute value $q_v^m$ for any complex embedding.  Since the action of $\Gamma_k$ on the $\Q_\ell(m)$ is unramified everywhere, and $Fr_v$ acts on it by $q_v^{-m}$, it follows that the eigenvalues of almost every $\rho_\ell(m)(Fr_v)$ have absolute value 1 in every complex embedding.  Thus the same is true of the eigenvalues of almost every $\rho_i(Fr_v)$.  Following the proof of \ref{easycor}, this implies that the central character of $\pi_i$ is unitary.

For the rest of the proof, fix a finite set of places $S$ of $k$ satisfying the following: If $v\notin S$, then $\rho_\ell(m)$ (and hence each $\rho_i$) is unramified at $v$, $X$ has good reduction at $v$, and (\ref{5a}) holds for all $i$.

The knowledge of almost every local $L$-factor $L_v(\pi_i,s)$ equivalent to knowing the parameters of almost every unramified local representation $\pi_{i,v}$ and so, by the strong multiplicity one theorem (applicable because $\pi_i$ has unitary central character), this knowledge determines $\pi_i$ up to isomorphism.  On the other hand, Chebotarev density (see \cite{Se}, loc. cit.) shows that knowledge of almost every local $L$-factor $L_v(\rho_i,s)$ determines $\rho_i$ up to isomorphism.  Hence the equalities in (\ref{5a}), which hold for all $v\notin S$, show that $\pi_i$ is trivial (i.e., $L_v(\pi_i,s)=1-q_v^{-s}$ for all $v$) if and only if $\rho_i$ is trivial (i.e., $L_v(\rho_i,s)=1-q_v^{-s}$ for all $v$).  So by Corollary \ref{ord} we get
\[\label{5b}
-\text{ord}_{s=1}L^S(\rho_i,s)=\begin{cases} 
1&\text{if } \rho_i \text{ trivial}\\ \tag{5b}
0&\text{if } \rho_i \text{ nontrivial}\end{cases}.
\]

Next we note that for $v\notin S$, the local $L$-factor $L_v(\rho_\ell(m),s)$ is the same whether we regard $\Gamma_k$ as acting on $V_\ell(m)$ or on $V_\ell(m)\otimes\bar \Q_\ell$. Thus for $v\notin S$ we have
\[
L_v(\rho_\ell(m),s)=\prod_i L_v(\rho_i,s),
\]
and hence
\[
L^S(\rho_\ell(m),s)=\prod_i L^S(\rho_i,s).
\]
By (\ref{5b}) this gives
\begin{eqnarray*}
-\text{ord}_{s=1}L^S(\rho_\ell(m),s) &=& -\sum_i \text{ord}_{s=1}L^S(\rho_i,s) \\
&\geq& \dim_{\bar\Q_\ell} (V_\ell(m)\otimes \bar\Q_\ell)^{\Gamma_k}\\
&=& \dim_{\Q_\ell} V_\ell(m)^{\Gamma_k} \\
&=& r_{\ell,k}^{(m)}.
\end{eqnarray*}
On the other hand, applying the Tate twist to $\rho_\ell$ has the effect of translation on its $L$-function, namely $L^S(\rho_\ell(m),s)=L^S(\rho_\ell,s+m)$.  Therefore,
\[
r_{\an,k}^{(m)} = -\text{ord}_{s=m+1}L^S(\rho_\ell,s)=-\text{ord}_{s=1}L^S(\rho_\ell(m),s)\geq r_{\ell,k}^{(m)}.
\]
Since we automatically have $r_{\alg,k}^{(m)}\leq r_{\ell,k}^{(m)}$, this completes the proof.
\end{proof}


\section{Remarks on the analogous question for number fields}\label{CM_sec}

The formulation of the Tate Conjecture in \S\ref{prelim_sec} for the case of global function fields also makes sense when $k$ is a number field, provided that the finite set of places $S$ also includes the archimedean ones.  One can then ask when the inequality $r_{\alg,k}^{(m)}\leq r_{\an,k}^{(m)}$ is known to hold.  In most cases where this is known to be true, such as some Shimura varieties for $m=1$ \cite{BR},\cite{K},\cite{MR} or Hilbert modular fourfolds for $m=2$ \cite{R1}, or certain $K3$ surfaces \cite{IS}, the full Tate Conjecture has actually been established.

If the Langlands conjectures for $\GL_n$ over number fields could be established, one could use the methods in this article to prove
\[\label{ineq}
r_{\alg,k}^{(m)}\leq r_{\ell,k}^{(m)}\leq r_{\an,k}^{(m)},\tag{6a}
\]
since Theorem \ref{JacSha} holds, in fact, for all global fields.  We remark, though, that this conjectural correspondence for number fields is not just a simple analogue of Theorem \ref{Lafforgue} and Corollary \ref{easycor}, due to the extra difficulties imposed by the places lying over $\ell$ and $\infty$.  One notable difference is that one must restrict attention to so-called \emph{algebraic} cuspidal representations \cite{Cl}.  In case $n=1$, this corresponds to A. Weil's notion of an \emph{idele class character of type $A_0$} \cite{Weil}.  This is an idele class character $\chi$ such that, if $v$ is archimedean, then $\chi_v(z)=z^{p_v}\bar z^{q_v}$; furthermore, we have $p_v+q_v=w$ for some $w\in\Z$ (the \emph{weight} of $\chi$) and all such $v$.

We note in passing that an idele class character of type $A_0$ with nonzero weight gives a counterexample to Lemma \ref{charlem} in the number field case, since no nonzero power would be trivial at the archimedean places.

Unfortunately, as it currently stands, the representation $\rho_\ell$ is known to correspond to an algebraic cuspidal representation in only a handful of cases.  Below we discuss one case where enough is known about $\rho_\ell$ to establish (\ref{ineq}).  Recall that an abelian variety $X$ over $k$ is said to be potentially CM-type if we can find a commutative semisimple algebra $\Lambda$ of dimension $2(\dim X)$ over $\Q$ and an isomorphism
\[
\theta\colon \Lambda \tilde{\To} \End_{\bar k}(X)\otimes \Q.
\]

\begin{prop}
Let $X$ be abelian variety over the number field $k$ which is potentially CM-type.  Then
\[
r_{\alg,k}^{(m)}\leq r_{\an,k}^{(m)}
\]
for all $0\leq m\leq \dim X$.
\end{prop}

\begin{proof}
Keeping the notation above, there is a finite Galois extension $L/k$ such that all elements of
\[
\theta(\Lambda)\cap \End_{\bar k} (X)
\]
are rational over $L$, and thus the action of $\Gamma_L$ on $H_{\text{\'et}}^1(X\times_k\bar k,\Q_\ell)$ is abelian.  Hence $\Gamma_L$ acts via a direct sum of characters if we extend scalars to $\bar\Q_\ell$, and these characters are associated to idele class characters of type $A_0$ in the sense given in the proof of Corollary \ref{easycor} \cite{ST}.

It is known that, as with any abelian variety, we have an isomorphism of $\Gamma_k$-modules
\[
H_{\text{\'et}}^{r}(X\times_k \bar k,\Q_\ell)\simeq \wedge^{r} H_{\text{\'et}}^1(X\times_k\bar k,\Q_\ell)
\]
(see \cite{M}, for instance).  Therefore the action of $\Gamma_L$ on $H_{\text{\'et}}^{r}(X\times_k \bar k,\Q_\ell)$ is also associated to idele class characters of type $A_0$ after extension to $\bar\Q_\ell$.

We focus on the case $r=2m$, letting $V_\ell=H_{\text{\'et}}^{2m}(X\times_k \bar k,\Q_\ell)$ and $\rho_\ell(m)\colon \Gamma_k\to \Aut_{\Q_\ell}(V_\ell(m))$ as in \S\ref{proof_sec}.  Then once again the semisimplification of the extension of $\rho_\ell(m)$ is a direct sum of irreducible $\bar\Q_\ell$-representations $M_i$ of $\Gamma_k$:
\[
(V_\ell(m)\otimes\bar\Q_\ell)^{\semisimp} = \bigoplus_i M_i.
\]
As before, let $\rho_i$ denote the $\bar\Q_\ell$-representation defined on $M_i$ by $\rho_\ell(m)$.

The observations above say that ${\rho_i}_{\vert_{\Gamma_L}}$ is a direct sum of characters associated to idele class characters of type $A_0$.  Due to this condition, a result of H. Yoshida (\cite{Y}, Theorem 1) gives a continuous finite-dimensional complex representation 
\[
r_i\colon W_{k}\to \Aut_\C(N_i)
\]
of the Weil group $W_k$ of $k$ and a finite set of places $S$ such that
\[
L_v(r_i,s)=L_v(\rho_i,s)
\]
for all $v\notin S$.  Yoshida's construction guarantees that $r_i$ is irreducible if and only if $\rho_i$ is irreducible (\cite{Y}, Theorems 1 and 2), so $r_i$ is irreducible.  We refer the reader to \cite{T2} for the notions of Weil groups, their representations, and the associated $L$-functions, as well as for facts listed below; for a very complete discussion of these matters, see \cite{De2}.

Following the same strategy as in the proof of Theorem 1.1, it suffices, for the completion of the proposition, to establish that
\[\label{JSanalog}
-\text{ord}_{s=1} L^{S}(r_i, s) = \dim_{\C} N_i^{W_k}.\tag{6b}
\]
We will do this by relating $r_i$ to characters of the Weil group, which are just idele class characters, and then using the analogue of Corollary \ref{ord} for the $L$-functions of such characters.

First we use the existence of a finite extension $E_i/k$ such that $r_i$ is the induction of a primitive representation of $W_{E_i}$.  (Here, primitive means that it is not induced from a smaller subgroup.)  In fact, one knows that $r_i=\text{Ind}_{E_i}^k (t_i\tensor \chi_i)$, where $t_i$ is a representation of $W_{E_i}$ of \emph{Galois type} and $\chi_i$ is a character of $W_{E_i}$. Thus $t_i$ is a representation of $W_{E_i}$ pulled back via the surjection $W_{E_i}\to \Gal(\bar k/E_i)$, while $\chi_i$ is an idele class character by virtue of the isomorphism $\mathbb A_{E_i}^\times/E_i^\times\simeq W_{E_i}^{ab}$.

Next, Brauer's induction theorem says that
\[
t_i \oplus \bigoplus_\alpha n_\alpha \text{Ind}_{F_\alpha}^{E_i} (\psi_\alpha) \simeq \bigoplus_\beta n_\beta' \text{Ind}_{F_\beta'}^{E_i} (\psi_\beta')
\]
for some finite extensions $F_\alpha/E_i$, $F_\beta'/E_i$, (idele class) characters $\psi_\alpha$, $\psi_\beta'$, and positive integers $n_\alpha$, $n_\beta'$.  (In other words, $t_i$ is a finite virtual sum of inductions of characters in the Grothendieck group.)  Since $\text{Ind}(\psi)\tensor \chi_i\simeq \text{Ind}(\psi\tensor\text{Res}\spc(\chi_i))$, we have
\[
\label{brauer}
(t_i\tensor\chi_i) \oplus \bigoplus_\alpha n_\alpha \text{Ind}_{F_\alpha}^{E_i} (\psi_\alpha\text{Res}_{F_\alpha}(\chi_i)) \simeq \bigoplus_\beta n_\beta' \text{Ind}_{F_\beta'}^{E_i} (\psi_\beta'\text{Res}_{F_\beta'}(\chi_i)).
\tag{6c}
\]
From this we conclude two things.

The first is that
\begin{eqnarray*}
L^S(r_i,s) &=& L^{S_i}(t_i\tensor\chi_i,s) \\ 
&=& \prod_\alpha L^{S_\alpha}(\psi_\alpha\text{Res}_{F_\alpha}(\chi_i),s)^{-n_\alpha}\prod_\beta L^{S_\beta'}(\psi_\beta'\text{Res}_{F_\beta'}(\chi_i),s)^{n_\beta'},
\end{eqnarray*}
where $S_i$ (resp., $S_\alpha$, $S_\beta'$) is the finite set of places in $E_i$ (resp., $F_\alpha$, $F_\beta'$) lying above those in $S$.  Hence we have
\begin{align*}
\label{6b}
-\text{ord}_{s=1} L^S(r_i,s) =& \sum_\alpha n_\alpha \text{ord}_{s=1} L^{S_\alpha}(\psi_\alpha\text{Res}_{F_\alpha}(\chi_i),s) \\
&- \sum_\beta n_\beta' \text{ord}_{s=1}L^{S_\beta'}(\psi_\beta'\text{Res}_{F_\beta'}(\chi_i),s).
\tag{6d}
\end{align*}
The $L$-functions on the right side of (\ref{6b}) are of the form $L^\Sigma(\omega, s)$ for some idele-class character $\omega$ and finite set of places $\Sigma$, and for such $L$-functions we have
\[
-\text{ord}_{s=1}L^\Sigma(\omega,s)=\begin{cases} 
1&\text{if } \omega=1\\ 
0&\text{if } \omega\neq 1\end{cases}.
\]
Hence (\ref{6b}) becomes
\begin{align*}
\label{6c}
-\text{ord}_{s=1} L^S(r_i,s) =& \sum_\beta n_\beta'
\left\{
\begin{array}{cl}
  1   & \text{if } \psi_\beta'\text{Res}_{F_\beta'}(\chi_i)=1 \\
  0   & \text{if } \psi_\beta'\text{Res}_{F_\beta'}(\chi_i)\neq 1
\end{array}
\right\} \\
& -
\sum_\alpha n_\alpha
\left\{
\begin{array}{cl}
  1   & \text{if } \psi_\alpha\text{Res}_{F_\alpha}(\chi_i)=1 \\
  0   & \text{if } \psi_\alpha\text{Res}_{F_\alpha}(\chi_i)\neq 1
\end{array}
\right\}.
\tag{6e}
\end{align*}

The second consequence of (\ref{brauer}) is that the dimension of the trivial representation in $t_i\tensor \chi_i$ (and in $r_i$ by induction) is equal to
\[
\sum_\beta n_\beta'
\left\{
\begin{array}{cl}
  1   & \text{if } \psi_\beta'\text{Res}_{F_\beta'}(\chi_i)=1 \\
  0   & \text{if } \psi_\beta'\text{Res}_{F_\beta'}(\chi_i)\neq 1
\end{array}
\right\}
-
\sum_\alpha n_\alpha
\left\{
\begin{array}{cl}
  1   & \text{if } \psi_\alpha\text{Res}_{F_\alpha}(\chi_i)=1 \\
  0   & \text{if } \psi_\alpha\text{Res}_{F_\alpha}(\chi_i)\neq 1
\end{array}
\right\},
\]
since $\text{Ind}(\omega)$ will contain the trivial representation only if $\omega=1$, and in that case it will occur with dimension one.  So putting this together with (\ref{6c}), we get (\ref{JSanalog}) as desired.
\end{proof}

\bibliography{tateinequality_arxiv_update}{}

\def\cprime{$'$}
\begin{thebibliography}{JPSS}

\bibitem[AT]{AT}
E.~Artin and J.~Tate.
\newblock {\em Class field theory}.
\newblock Advanced Book Classics. Addison-Wesley Publishing Company Advanced
  Book Program, Redwood City, CA, second edition, 1990.

\bibitem[BR]{BR}
D.~Blasius and J.~Rogawski.
\newblock {Tate classes and arithmetic quotients of the two-ball}.
\newblock In {\em The zeta functions of Picard modular surfaces}, pages
  421--444. Univ. Montr\'eal, Montreal, QC, 1992.

\bibitem[Clo]{Cl}
L.~Clozel.
\newblock {Motifs et formes automorphes: applications du principe de
  fonctorialit\'e}.
\newblock In {\em Automorphic forms, {S}himura varieties, and {$L$}-functions,
  {V}ol.\ {I} ({A}nn {A}rbor, {MI}, 1988)}, volume~10 of {\em Perspect. Math.},
  pages 77--159. Academic Press, Boston, MA, 1990.

\bibitem[Del1]{De2}
P.~Deligne.
\newblock {Les constantes des \'equations fonctionnelles des fonctions {$L$}}.
\newblock In {\em Modular functions of one variable, II (Proc. Internat. Summer
  School, Univ. Antwerp, Antwerp, 1972)}, pages 501--597. Lecture Notes in
  Math., Vol. 349. Springer, Berlin, 1973.

\bibitem[Del2]{De1}
P.~Deligne.
\newblock {La conjecture de {W}eil. {I}}.
\newblock {\em Inst. Hautes \'Etudes Sci. Publ. Math.} (1974), 273--307.

\bibitem[Dri1]{Dr1}
V.~G. Drinfel{\cprime}d.
\newblock {Cohomology of compactified moduli varieties of {$F$}-sheaves of rank
  {$2$}}.
\newblock {\em Zap. Nauchn. Sem. Leningrad. Otdel. Mat. Inst. Steklov. (LOMI)}
  {\bf 162} (1987), 107--158, 189.

\bibitem[Dri2]{Dr2}
V.~G. Drinfel{\cprime}d.
\newblock {Proof of the {P}etersson conjecture for {${\rm GL}(2)$} over a
  global field of characteristic {$p$}}.
\newblock {\em Funktsional. Anal. i Prilozhen.} {\bf 22} (1988), 34--54, 96.

\bibitem[JPSS]{JPSS}
H.~Jacquet, I.~I. Piatetskii-Shapiro, and J.~A. Shalika.
\newblock {Rankin-{S}elberg convolutions}.
\newblock {\em Amer. J. Math.} {\bf 105} (1983), 367--464.

\bibitem[JS]{JS}
H.~Jacquet and J.~A. Shalika.
\newblock {On {E}uler products and the classification of automorphic forms. {I}
  and {II}}.
\newblock {\em Amer. J. Math.} {\bf 103} (1981), 499--558 and 777--815.

\bibitem[Kli]{K}
C.~Klingenberg.
\newblock {Die {T}ate-{V}ermutungen f\"ur {H}ilbert-{B}lumenthal-{F}l\"achen}.
\newblock {\em Invent. Math.} {\bf 89} (1987), 291--317.

\bibitem[Laf]{Laf}
L.~Lafforgue.
\newblock {Chtoucas de {D}rinfeld et correspondance de {L}anglands}.
\newblock {\em Invent. Math.} {\bf 147} (2002), 1--241.

\bibitem[Lau]{Lau}
G.~Laumon.
\newblock {La correspondance de {L}anglands sur les corps de fonctions
  (d'apr\`es {L}aurent {L}afforgue)}.
\newblock {\em Ast\'erisque} (2002), 207--265.
\newblock S{\'e}minaire Bourbaki, Vol. 1999/2000.

\bibitem[Mil]{Mil}
J.~S. Milne.
\newblock {\em \'{E}tale cohomology}, volume~33 of {\em Princeton Mathematical
  Series}.
\newblock Princeton University Press, Princeton, N.J., 1980.

\bibitem[Mum]{M}
D.~Mumford.
\newblock {\em Abelian varieties}.
\newblock Tata Institute of Fundamental Research Studies in Mathematics, No. 5.
  Published for the Tata Institute of Fundamental Research, Bombay, 1970.

\bibitem[MR]{MR}
V.~K. Murty and D.~Ramakrishnan.
\newblock {Period relations and the {T}ate conjecture for {H}ilbert modular
  surfaces}.
\newblock {\em Invent. Math.} {\bf 89} (1987), 319--345.

\bibitem[Ram]{R1}
D.~Ramakrishnan.
\newblock {Algebraic cycles on {H}ilbert modular fourfolds and poles of
  {$L$}-functions}.
\newblock In {\em Algebraic groups and arithmetic}, pages 221--274. Tata Inst.
  Fund. Res., Mumbai, 2004.

\bibitem[Ser]{Se}
J.-P. Serre.
\newblock {\em Abelian {$l$}-adic representations and elliptic curves}.
\newblock Advanced Book Classics. Addison-Wesley Publishing Company Advanced
  Book Program, Redwood City, CA, second edition, 1989.
\newblock With the collaboration of Willem Kuyk and John Labute.

\bibitem[ST]{ST}
G.~Shimura and Y.~Taniyama.
\newblock {\em Complex multiplication of abelian varieties and its applications
  to number theory}, volume~6 of {\em Publications of the Mathematical Society
  of Japan}.
\newblock The Mathematical Society of Japan, Tokyo, 1961.

\bibitem[SI]{IS}
T.~Shioda and H.~Inose.
\newblock {On singular {$K3$} surfaces}.
\newblock In {\em Complex analysis and algebraic geometry}, pages 119--136.
  Iwanami Shoten, Tokyo, 1977.

\bibitem[Tat1]{T1}
J.~Tate.
\newblock {Algebraic cycles and poles of zeta functions}.
\newblock In {\em Arithmetical Algebraic Geometry (Proc. Conf. Purdue Univ.,
  1963)}, pages 93--110. Harper \& Row, New York, 1965.

\bibitem[Tat2]{T2}
J.~Tate.
\newblock {Number theoretic background}.
\newblock In {\em Automorphic forms, representations and $L$-functions (Proc.
  Sympos. Pure Math., Oregon State Univ., Corvallis, Ore., 1977), Part 2},
  Proc. Sympos. Pure Math., XXXIII, pages 3--26. Amer. Math. Soc., Providence,
  R.I., 1979.

\bibitem[Wei]{Weil}
A.~Weil.
\newblock {On a certain type of characters of the id\`ele-class group of an
  algebraic number-field}.
\newblock In {\em Proceedings of the international symposium on algebraic
  number theory, {T}okyo \& {N}ikko, 1955}, pages 1--7, Tokyo, 1956. Science
  Council of Japan.

\bibitem[Yos]{Y}
H.~Yoshida.
\newblock {Abelian varieties with complex multiplication and representations of
  the {W}eil groups}.
\newblock {\em Ann. of Math. (2)} {\bf 114} (1981), 87--102.

\end{thebibliography}
\bibliographystyle{math}

\end{document}